\documentclass[11pt]{amsart}
\usepackage[T1]{fontenc}
\usepackage{lmodern}
\usepackage{amsmath,amssymb,amsthm,mathtools}
\usepackage{microtype}
\usepackage[a4paper,margin=30mm]{geometry}
\usepackage{enumitem}
\usepackage{xcolor}
\usepackage[colorlinks=true,linkcolor=blue!50!black,citecolor=blue!50!black,urlcolor=blue!50!black]{hyperref}
\hypersetup{pdftitle={Quasidiagonal traces need not form a face},pdfauthor={Mehdi Moradi}}
\newtheorem{theorem}{Theorem}[section]
\newtheorem{lemma}[theorem]{Lemma}

\theoremstyle{definition}

\theoremstyle{remark}

\DeclareMathOperator{\Tr}{Tr}
\DeclareMathOperator{\tr}{tr}

\numberwithin{equation}{section}
\title[Quasidiagonal traces]{Quasidiagonal traces need not form a face}
\author{Mehdi Moradi}
\address{Department of Mathematics, University of Toronto, Toronto, Ontario, Canada M5S 2E4}
\email{mehdi.moradi@utoronto.ca}
\date{September 15, 2026}
\subjclass[2020]{Primary 46L05, 46L35; Secondary 22D25}
\keywords{Quasidiagonal trace, amenable trace, residually finite-dimensional
\(C^*\)-algebra, Kazhdan's property \((T)\), tracial state space}
\begin{document}
\begin{abstract}
Let \(G\) be an infinite residually finite countable discrete group
with Kazhdan's property~\((T)\).
Assume that its finite-dimensional irreducible unitary representations
admit an exhaustive ordering with nondecreasing dimensions tending to
infinity and bounded consecutive dimension ratios.
From this data we construct a separable unital residually
finite-dimensional \(C^*\)-algebra with a faithful quasidiagonal
tracial state \(\tau=\frac14\mu_1+\frac34\mu_2\), where \(\mu_1\)
is not quasidiagonal. For the resulting algebra, quasidiagonal traces
do not form a face of the tracial state space.
The construction has two copies of a weighted representation ladder
joined at their top levels. We prove that every quasidiagonal trace
gives equal weight to the two copies. The proof combines uniform
Kazhdan spectral projections, exact ranks of rounded finite-rank
compressions, and a weighted Hilbert--Schmidt comparison that controls
all representation blocks.
\end{abstract}
\maketitle
\section{Introduction}
Finite-dimensional approximation distinguishes several natural subsets
of the tracial state space of a \(C^*\)-algebra. Amenable traces admit
matrix models whose multiplicativity defects vanish in normalized
Hilbert--Schmidt norm. Quasidiagonal traces require those defects to
vanish in operator norm. This distinction concerns the models
themselves, even when both notions use the same pointwise convergence
of normalized matrix traces.

Brown's systematic treatment~\cite{Brown2006} identifies these
approximation properties and their permanence rules. Both sets of
traces are convex and weak-star closed, and the amenable traces form
a face. The latter property says that a convex decomposition of an
amenable trace has amenable components. Under the representation-growth
hypotheses below, we construct a quasidiagonal trace for which the
corresponding conclusion fails.

\begin{theorem}\label{thm:main}
Let \(G\) be an infinite residually finite countable discrete group
with property~\((T)\).
Suppose that \((\pi_k,H_k)_{k\geq0}\) lists every equivalence class
of finite-dimensional irreducible unitary representations exactly once.
Write \(d_k=\dim H_k\), and assume
\[
\begin{gathered}
 \pi_0=1_G,\qquad 1=d_0\leq d_1\leq\cdots,\qquad
 d_k\longrightarrow\infty,\\
 R_0:=\sup_{k\geq1}\frac{d_k}{d_{k-1}}<\infty.
\end{gathered}
\]
For every integer \(b>R_0\) and every choice of Hilbert-space
isometries \(J_k:H_{k-1}\to H_k\), the algebra \(A\) constructed
in Section~\ref{sec:construction} is separable, unital and residually
finite-dimensional. It admits a faithful quasidiagonal tracial state
\(\tau\) and tracial states \(\mu_1,\mu_2\) such that
\[
                 \tau=\frac14\mu_1+\frac34\mu_2,
                 \qquad \mu_1\notin T_{\mathrm{qd}}(A).
\]
Consequently \(T_{\mathrm{qd}}(A)\) is not a face of \(T(A)\).
\end{theorem}

The bounded-ratio hypothesis is the additional slow dimension growth
assumption. The group hypotheses already ensure that the complete
irreducible dual can be ordered as above with \(d_k\to\infty\).
We justify this point to distinguish an exhaustive sequence from a
selected family of representations.

For each fixed dimension \(d\), choose a finite Kazhdan set \(F\)
and a Kazhdan constant \(\kappa>0\). If two irreducible
representations \(\pi,\rho\) on \(\mathbb C^d\) satisfy
\(\max_{g\in F}\|\pi(g)-\rho(g)\|<\kappa\), then
\(d^{-1/2}I\) is a Kazhdan almost-invariant unit vector for the
Hilbert--Schmidt representation \(X\mapsto\pi(g)X\rho(g)^*\).
A nonzero invariant vector is an intertwiner, and irreducibility
and polar decomposition imply \(\pi\simeq\rho\).
One matrix realization per inequivalent class therefore gives a
\(\kappa\)-separated subset of the compact space \(U(d)^F\).
There are only finitely many such classes.

There are infinitely many irreducible classes with finite image.
Otherwise their kernels would have a finite-index intersection.
Residual finiteness, together with decomposition of the regular
representations of finite quotients, makes that intersection trivial,
forcing \(G\) to be finite. Thus the complete finite-dimensional
irreducible dual is infinite, and finiteness in each dimension permits
an exhaustive nondecreasing ordering with \(d_k\to\infty\).
Property~\((T)\) will also supply the uniform spectral gap in the
trace-balance argument. The RFD property of the constructed algebra
itself follows directly from its coordinate representations.

The construction has two copies of a finite representation ladder.
A partial isometry joins their top levels. The top levels have
vanishing normalized trace, so the limiting trace splits into
components supported on the separate copies. Nevertheless, the top
join imposes an exact rank constraint on sufficiently accurate
operator-norm matrix models. Kazhdan's property~\((T)\) makes the
required estimates uniform over every finite-dimensional irreducible
representation. A weighted Hilbert--Schmidt comparison then turns
the rank constraint into equality of the two copy weights for every
quasidiagonal trace.

Two features require explicit treatment. The representation sequence
must exhaust the finite-dimensional irreducible dual; a selected
subsequence with well-controlled dimensions is insufficient for the
rank comparison. In addition, equality on any fixed collection of
representation blocks does not imply equality of the whole trace.
The proof below accounts for every block with a uniform estimate and
an exact finite-rank cutoff.

\subsection{Relation to earlier work}
Voiculescu's abstract characterization of quasidiagonality
\cite{Voiculescu1991} uses asymptotically multiplicative and
asymptotically isometric completely positive matrix approximations.
For a quasidiagonal trace only the prescribed trace convergence is
required; no asymptotic isometry condition is imposed on its witnesses.
Our faithful trace is obtained using genuine finite-dimensional
representations, while its non-quasidiagonal component is detected
by testing all completely positive witnesses.

Among the early antecedents of this construction is Wassermann's
example of a separable quasidiagonal \(C^*\)-algebra whose image in
the Calkin algebra is not quasidiagonal~\cite{WassermannCalkin1991}.
Brown's memoir~\cite[Section~6.5, especially Corollaries~6.5.8
and~6.5.11]{Brown2006} explains this circle of examples through
products of matrix algebras, their \(c_0\)-ideals, and obstructions
coming from traces and finite-dimensional representations.
This provides an early model for the passage from finite-dimensional
data to a quotient in which quasidiagonality fails.
Wassermann's property-\((T)\) construction~\cite{Wassermann1991}
also uses the isolation of finite-dimensional representations.
We use invariant-vector projections in their tensor products,
the associated uniform spectral gap, and rank stability.
Background on property~\((T)\) and
finite-dimensional approximation is developed in
\cite{BekkaDeLaHarpeValette2008,BrownOzawa2008}.

The representation ladders and weighted identity-vector shifts are
closely related to Ozawa's construction~\cite{Ozawa2026}.
Here the geometry has two copies joined at the top, and the
conclusion concerns domination between traces on an RFD algebra.
The proof is given directly in terms of finite-rank compressions;
it does not use non-quasidiagonality of the hyperfinite factor as
an input.

The positive results of Tikuisis, White and Winter
\cite[Theorem~A]{TikuisisWhiteWinter2017} show that faithful traces on
separable nuclear \(C^*\)-algebras satisfying the UCT are quasidiagonal.
Gabe~\cite{Gabe2017} extends this conclusion to faithful amenable
traces on separable exact UCT algebras; for separable exact
quasidiagonal UCT algebras, all amenable traces are quasidiagonal.
These hypotheses describe important settings in which amenability
and norm approximation are closely linked. The present construction
imposes neither nuclearity nor the UCT.

For comparison with the face property, every component of the trace
\(\tau\) in Theorem~\ref{thm:main} is amenable, since quasidiagonal
traces are amenable and amenable traces form a face
\cite[Proposition~3.5.3]{Brown2006}. The obstruction proved here is
therefore specifically an obstruction to operator-norm
multiplicativity.

\subsection{Notation}
All algebras in the main construction are complex, unital and
separable. Write \(M_n=B(\mathbb C^n)\),
\(\Tr_n\) for the unnormalized matrix trace, and
\(\tr_n=n^{-1}\Tr_n\). The Hilbert--Schmidt norm on operators on an
arbitrary Hilbert space is always unnormalized and denoted by
\(\|\cdot\|_{\mathrm{HS}}\); a normalized matrix \(2\)-norm is
written \(\|\cdot\|_{2,\tr_n}\).
An overline denotes the conjugate Hilbert space or conjugate
representation. Hilbert-space tensor products use their natural
Hilbert-space norm, and matrix amplifications carry their unique
\(C^*\)-norm.

For a unital \(C^*\)-algebra \(B\), let \(T(B)\) be its tracial state
space. A trace \(\lambda\) lies in \(T_{\mathrm{qd}}(B)\) when there
are unital completely positive maps \(\psi_j:B\to M_{N_j}\) with
\[
 \|\psi_j(ab)-\psi_j(a)\psi_j(b)\|\to0,\qquad
 \tr_{N_j}(\psi_j(a))\to\lambda(a)\quad(a,b\in B).
\]
Replacing the first norm by the normalized \(2\)-norm defines
\(T_{\mathrm{am}}(B)\). The algebra is RFD if its
finite-dimensional representations separate points.
A convex subset \(F\) of \(T(B)\) is a face if
\(t\lambda_1+(1-t)\lambda_2\in F\), with \(0<t<1\), implies
\(\lambda_1,\lambda_2\in F\).

\section{The matrix construction}\label{sec:construction}
Let \(G\) have the hypotheses of Theorem~\ref{thm:main}. Assume that
\((\pi_k,H_k)_{k\geq0}\) lists all equivalence classes of finite-dimensional
irreducible unitary representations of \(G\), once each, with
\[
\begin{gathered}
 \pi_0=1_G,\qquad d_k=\dim H_k,\qquad 1=d_0\leq d_1\leq\cdots,\\
 d_k\longrightarrow\infty,\qquad
 R_0:=\sup_{k\geq1}\frac{d_k}{d_{k-1}}<\infty.
\end{gathered}
\]
The construction below uses ordinary Hilbert-space isometries
\(J_k:H_{k-1}\to H_k\), not intertwining maps.
Fix an integer \(b>R_0\), and put \(V_\ell=(\mathbb C^b)^{\otimes\ell}\),
where \(V_0=\mathbb C\). For \(1\leq i\leq b\), let
\[
 \Delta_{\ell,i}:V_{\ell+1}\longrightarrow V_\ell,\qquad
 \Delta_{\ell,i}(\delta_j\otimes v)=\delta_{ij}v.
\]
Thus \(\Delta_{\ell,i}\Delta_{\ell,j}^*=\delta_{ij}1\) and
\(\sum_i\Delta_{\ell,i}^*\Delta_{\ell,i}=1\).
Write
\[
 L_m=\bigoplus_{k=0}^m\overline H_k\otimes V_{m-k},
 \qquad E_m=L_m^{(0)}\oplus L_m^{(1)},\qquad
 S_m=\sum_{k=0}^m d_kb^{-k}.
\]
Then
\[
 S_m\uparrow S:=\sum_{k=0}^{\infty}d_kb^{-k}<\infty,\qquad
 c_m:=\dim E_m=2b^mS_m,\qquad d_mb^{-m}\longrightarrow0.
\]
Indeed \(d_k\leq R_0^k\), including the first ratio.

Let \(e_m^\varepsilon\) project onto \(L_m^{(\varepsilon)}\);
let \(F_m^\varepsilon\) project onto the \(k=m\) summand of that copy.
The representation
\[
 \sigma_m(g)=
 \bigoplus_{\varepsilon=0}^1\bigoplus_{k=0}^m
       \overline\pi_k(g)\otimes1_{V_{m-k}}
\]
acts on \(E_m\). Define \(T_{m,i}^\varepsilon\) to be zero on the other
copy and on its own top summand, and, for \(1\leq k\leq m\), set
\[
 T_{m,i}^\varepsilon\big|_{\overline H_{k-1}\otimes V_{m-k+1}}
       =\overline J_k\otimes\Delta_{m-k,i}.
\]
These are contractions. Finally, \(s_m\) is the canonical identity
from the top summand of copy~1 onto the top summand of copy~0,
extended by zero. In particular
\[
 s_m^*s_m=F_m^1,\qquad s_ms_m^*=F_m^0.
\]
Sequences of these operators are denoted without the subscript \(m\).
Put
\[
 \mathcal K=\bigoplus_{m=0}^{\infty}B(E_m),\qquad
 A=C^*\bigl(1,\mathcal K,e^\varepsilon,F^\varepsilon,
               T_i^\varepsilon,\sigma(g),s\bigr)
       \subseteq\prod_{m=0}^{\infty}B(E_m),
\]
where both values \(\varepsilon=0,1\), all \(1\leq i\leq b\), and
all \(g\in G\) are included among the generators.
Here \(\mathcal K\) is the norm-\(c_0\) direct sum. The algebra \(A\)
is unital and separable; its coordinate representations
\(\rho_m:A\to B(E_m)\) separate points, so \(A\) is residually
finite-dimensional.

\section{A faithful quasidiagonal trace and its components}
\label{sec:traces}
Set \(\tau_m=\tr_{c_m}\circ\rho_m\). Fix a free ultrafilter \(\nu\)
on \(\mathbb N_0\), and define
\[
 \tau_\infty=\lim_{m\to\nu}\tau_m,\qquad
 \gamma=\sum_{m=0}^{\infty}2^{-m-1}\tau_m,\qquad
 \tau=\frac12(\tau_\infty+\gamma).
\]
The convergence defining \(\gamma\) is in the norm of \(A^*\).
Every coefficient is positive, and the coordinate representations
separate points, so \(\gamma\), and therefore \(\tau\), is faithful.

\begin{lemma}\label{lem:hom-models}
The trace \(\tau\) is a pointwise limit of normalized traces of
finite-dimensional unital representations of \(A\). In particular it
is quasidiagonal.
\end{lemma}
\begin{proof}
Let \(\mathcal D\) be the normalized traces of finite direct sums
of the coordinate representations \(\rho_m\), allowing repetitions,
and let \(K\) be its weak-star closure in \(T(A)\).
Every rational convex combination of coordinate traces belongs to
\(\mathcal D\). Indeed, for \(r_i=p_i/q\geq0\) with
\(\sum_{i=1}^n p_i=q\), set
\[
 D=q\prod_{i=1}^n c_{m_i},\qquad
 h_i=\frac{r_iD}{c_{m_i}}\in\mathbb N_0.
\]
Omitting zero multiplicities, the unital representation
\[
 \bigoplus_{i=1}^n\rho_{m_i}^{\oplus h_i}:A\longrightarrow M_D
\]
has normalized trace \(\sum_i r_i\tau_{m_i}\).
Conversely, the normalized trace of any such direct sum is a rational
convex combination of coordinate traces, with weights equal to the
summand dimensions divided by the total dimension. Thus
\(\mathcal D\) is closed under rational convex combinations.

To see that \(K\) is convex, fix finitely many traces in \(K\),
convex weights, and a finite set of elements of \(A\).
Approximate each trace on that set by an element of \(\mathcal D\),
and approximate the weights by nonnegative rational weights summing
to one. The resulting rational combination belongs to \(\mathcal D\)
and approximates the desired mixture on the given set: the first
error is bounded by the largest trace-approximation error, and the
second by \(\max\|a\|\) times the sum of the weight errors.
Both errors can be made arbitrarily small, proving convexity.

The ultralimit \(\tau_\infty\) belongs to \(K\), since every
weak-star neighborhood of it contains a coordinate trace.
Furthermore,
\[
 \gamma_M=\sum_{m=0}^M2^{-m-1}\tau_m+2^{-M-1}\tau_0
       \in\mathcal D,\qquad
 \|\gamma_M-\gamma\|\leq2^{-M}.
\]
Hence \(\gamma\in K\), and convexity gives \(\tau\in K\).
Finally choose a norm-dense sequence \((a_\ell)\) in the unit ball
of the separable algebra \(A\). For each \(j\), membership in \(K\)
gives a finite direct sum representation \(\theta_j:A\to M_{N_j}\)
whose normalized trace differs from \(\tau\) by less than \(1/j\)
on \(a_1,\ldots,a_j\). Since these traces and \(\tau\) have norm
one, density gives pointwise convergence on all of \(A\).
The representations are unital completely positive and their
multiplicativity defects are identically zero, so they witness
quasidiagonality.
\end{proof}

For \(\varepsilon=0,1\), consider the states
\[
 \lambda_m^\varepsilon(a)=2\tau_m(e^\varepsilon a),\qquad
 \mu_\varepsilon=\lim_{m\to\nu}\lambda_m^\varepsilon .
\]
Positivity follows by replacing \(e_m^\varepsilon\rho_m(a)\) under
the matrix trace with
\(e_m^\varepsilon\rho_m(a)e_m^\varepsilon\).

\begin{lemma}\label{lem:components}
The functionals \(\mu_0,\mu_1\) are tracial states and satisfy
\[
 \tau_\infty=\tfrac12(\mu_0+\mu_1),\qquad
 \mu_\varepsilon(e^\delta)=\delta_{\varepsilon\delta}.
\]
Hence, with \(\eta=\frac13\mu_0+\frac23\gamma\),
\[
                 \tau=\frac14\mu_1+\frac34\eta .
\]
\end{lemma}
\begin{proof}
For every \(a\in A\),
\[
 \|[e_m^\varepsilon,\rho_m(a)]\|_{2,\tr_{c_m}}
       \longrightarrow0 .
\]
To verify this, all the specified generators other than \(s\) and
\(\mathcal K\) commute with \(e^\varepsilon\). The \(c_0\)-summands
have coordinate norms tending to zero. For \(s\) and \(s^*\) the
squared commutator norm is
\[
 \frac{d_m}{c_m}=\frac{d_m}{2b^mS_m}\longrightarrow0.
\]
The product rule for commutators and norm density give the assertion
for every \(a\). Matrix trace cyclicity and Cauchy--Schwarz now give
\[
 |\lambda_m^\varepsilon(ab-ba)|
     =2|\tau_m([e^\varepsilon,a]b)|
     \leq 2\|[e_m^\varepsilon,\rho_m(a)]\|_2\,\|b\|
     \longrightarrow0.
\]
Thus \(\mu_\varepsilon\) is tracial. Since \(e^0+e^1=1\),
\(\lambda_m^0+\lambda_m^1=2\tau_m\).
The asserted values on \(e^\delta\) follow from orthogonality and
the equal dimensions of the two copies.
The final convex decomposition follows by substitution.
\end{proof}


\section{Equality of the two sector masses}
\label{sec:direct-rank-proof}

We retain the complete enumeration $(\pi_k,H_k)_{k\geq0}$ of
pairwise inequivalent finite-dimensional irreducible representations of
$G$, with $\pi_0=1_G$ and $d_k=\dim H_k$, and all the data defining $A$.
In particular, the embeddings $J_k:H_{k-1}\to H_k$ are isometries, and
the bridges $s_m$ are the canonical top-level identifications specified
in Section~\ref{sec:construction}. Put
\[
 D=R_0=\sup_{k\geq1}\frac{d_k}{d_{k-1}}<\infty.
\]
All traces denoted by $\operatorname{Tr}$ in this section are
unnormalized. Hilbert--Schmidt norms are taken with respect to these
traces.

\begin{theorem}\label{thm:direct-sector-equality}
Every quasidiagonal tracial state $\mu$ on the algebra $A$ satisfies
$\mu(e^0)=\mu(e^1)=1/2$.
\end{theorem}

We first record the finite-dimensional compression estimate needed to
extract exact integer identities from norm approximations.

\begin{lemma}\label{lem:direct-compression-rank}
Let $v:\mathcal H_1\to\mathcal H_2$ be a partial isometry with
$v^*v=e$ and $vv^*=f$. Let $P_i$ be finite-rank projections on
$\mathcal H_i$. If
\[
 \|P_2v-vP_1\|\leq\eta<1/100,
\]
then
\[
 \operatorname{rank}\chi_{[1/2,\infty)}(P_1eP_1)
 =\operatorname{rank}\chi_{[1/2,\infty)}(P_2fP_2).
\]
Here each compression is regarded on the range of the corresponding
finite-rank projection.
\end{lemma}

\begin{proof}
The hypothesis and its adjoint imply
$\|[P_1,e]\|,\|[P_2,f]\|\leq2\eta$.
Set $a=P_1eP_1$, $b=P_2fP_2$, and let $E,F$ be their respective
spectral projections for $[1/2,\infty)$. The off-diagonal corner formula
for the compression of a projection gives
\[
 \|a-a^2\|,\ \|b-b^2\|\leq4\eta^2.
\]
For $t\in[0,1]$, the inequality
$\operatorname{dist}(t,\{0,1\})\leq2t(1-t)$ therefore gives
\[
 \|a-E\|,\ \|b-F\|\leq8\eta^2.
\]
In particular, the spectra are uniformly separated from $1/2$.
Put $W_0=P_2vP_1$. Then
\[
 a-W_0^*W_0=P_1v^*(1-P_2)vP_1,
 \qquad
 b-W_0W_0^*=P_2v(1-P_1)v^*P_2.
\]
Both positive operators have norm at most $\eta^2$; the second estimate
uses the adjoint of the assumed rectangular commutator bound. Hence
\[
 \|W_0^*W_0-E\|,\ \|W_0W_0^*-F\|\leq9\eta^2.
\]
It follows that $\|W_0(1-E)\|,\|(1-F)W_0\|\leq3\eta$.
For $W=FW_0E$, this gives $\|W-W_0\|\leq6\eta$. Since $W$ and $W_0$
are contractions,
\[
 \|W^*W-E\|,\ \|WW^*-F\|\leq12\eta+9\eta^2<1.
\]
Thus the two Gram operators are invertible on $E\mathcal H_1$ and
$F\mathcal H_2$. The polar decomposition of $W$ has initial projection
$E$ and final projection $F$, proving equality of the ranks.
\end{proof}

If $e_1,\ldots,e_l$ are mutually orthogonal projections on
$\mathcal H$, the column $v=(e_1,\ldots,e_l)^t:\mathcal H\to
\mathcal H^{\oplus l}$ has initial projection $\sum_i e_i$ and final
projection $\operatorname{diag}(e_1,\ldots,e_l)$. Moreover,
\[
 \|\operatorname{diag}(P,\ldots,P)v-vP\|
 \leq\left(\sum_i\|[P,e_i]\|^2\right)^{1/2}.
\]
Consequently Lemma~\ref{lem:direct-compression-rank} proves additivity
of rounded compressed ranks whenever this last bound is below $1/100$.

\begin{proof}[Proof of Theorem~\ref{thm:direct-sector-equality}]
\emph{Uniform spectral projections.}
Choose a finite symmetric Kazhdan set $\mathcal F\subset G$ and a
constant $\kappa>0$ such that, in every unitary representation $U$,
\begin{equation}\label{eq:direct-kazhdan}
 \max_{g\in\mathcal F}\|U(g)\xi-\xi\|\geq\kappa\|\xi\|
 \quad\text{if }\xi\perp\mathcal H^U.
\end{equation}
All the corresponding unitaries $\sigma(g)$ belong to $A$.
Write $h=|\mathcal F|$, $\alpha=\kappa^2/(2h)$, and $C_q=2/\alpha$.

Let $\rho:A\to B(\mathcal H)$ be a unital representation and put
$u(g)=\rho(\sigma(g))$. Define, on $H_k\otimes\mathcal H$,
\[
 \mathsf L_k=\frac1h\sum_{g\in\mathcal F}
        (1-\pi_k(g)\otimes u(g)),
 \qquad q_k=\chi_{\{0\}}(\mathsf L_k).
\]
Symmetry and \eqref{eq:direct-kazhdan} show that
\[
 \langle \mathsf L_k\xi,\xi\rangle
 =\frac1{2h}\sum_{g\in\mathcal F}
       \|(\pi_k(g)\otimes u(g))\xi-\xi\|^2,
 \qquad \operatorname{sp}(\mathsf L_k)\subset\{0\}\cup[\alpha,2].
\]
Thus $q_k$ projects onto the invariant vectors of $\pi_k\otimes u$.
The same uniform spectral gap in all representations of $A$ makes
$q_k$ the image of a projection in $M_{d_k}(A)$.
If a projection $P$ on $\mathcal H$ satisfies
$\max_{g\in\mathcal F}\|[P,u(g)]\|\leq\delta$, then
$\|[1\otimes P,\mathsf L_k]\|\leq\delta$. Resolvent integration around
$|z|=\alpha/2$ gives the dimension-independent bound
\begin{equation}\label{eq:direct-q-commutator}
 \|[1\otimes P,q_k]\|
 \leq\frac1{2\pi}(\pi\alpha)(2/\alpha)^2\delta
 =C_q\delta.
\end{equation}

Let $z_k$ be the orthogonal projection onto the
$\overline\pi_k$-isotypic subspace of $u$. No membership of $z_k$ in
$\rho(A)$ is asserted. Write that subspace as
$\overline H_k\otimes\mathcal M_k$. For an orthonormal basis
$(\xi_a)_{a=1}^{d_k}$ of $H_k$, set
\[
 \Omega_k=d_k^{-1/2}\sum_{a=1}^{d_k}
                    \xi_a\otimes\overline\xi_a.
\]
Invariant vectors of $\pi_k\otimes u$ correspond to intertwiners
from $\overline\pi_k$ to $u$. Schur's lemma therefore gives
\begin{equation}\label{eq:direct-partial-trace}
 q_k=|\Omega_k\rangle\langle\Omega_k|\otimes1_{\mathcal M_k}
 \quad\text{on }H_k\otimes(\overline H_k\otimes\mathcal M_k),
 \qquad
 (\operatorname{Tr}_{H_k}\otimes\mathrm{id})(q_k)=z_k/d_k.
\end{equation}
The projection vanishes on the other isotypic subspaces and on the
orthogonal complement of all finite-dimensional subrepresentations.
Indeed, a nonzero intertwiner from $\overline\pi_k$ would supply such
a subrepresentation.

The hypothesis $d_k\to\infty$ is part of the admissible data in
Theorem~\ref{thm:main}.

\smallskip\noindent\emph{Rectangular shift identities.}
We next establish the exact algebraic relations used for the ranks.
Put
\[
 q_k^\varepsilon=q_k(1\otimes\rho(e^\varepsilon)),
 \qquad b_k^\varepsilon=q_k(1\otimes\rho(F^\varepsilon)).
\]
These are projections with $b_k^\varepsilon\leq q_k^\varepsilon$.
At coordinate $m\geq k$, the range of $q_k^\varepsilon$ is
$\mathbb C\Omega_k\otimes V_{m-k}$ in
$H_k\otimes E_m^\varepsilon$; it is zero if $m<k$.
The projection $b_k^\varepsilon$ has the same range at coordinate $m=k$
and vanishes at every other coordinate.

For $k\geq1$, define the rectangular operator
\[
 V_{k,i}^\varepsilon
 =\sqrt{\frac{d_k}{d_{k-1}}}\,
 q_k^\varepsilon(J_k\otimes\rho(T_i^\varepsilon))q_{k-1}^\varepsilon:
 H_{k-1}\otimes\mathcal H\longrightarrow H_k\otimes\mathcal H.
\]
The external factor is $J_k:H_{k-1}\to H_k$, while the internal
source shift has component
\[
 \overline J_k\otimes\Delta_{m-k,i}:
 \overline H_{k-1}\otimes V_{m-k+1}
 \longrightarrow\overline H_k\otimes V_{m-k}.
\]
The conjugate embeddings are essential to the scalar identity
\[
 \langle\Omega_k,(J_k\otimes\overline J_k)\Omega_{k-1}\rangle
 =\sqrt{d_{k-1}/d_k}.
\]
Hence $V_{k,i}^\varepsilon$ acts as $\Delta_{m-k,i}$ on the selected
invariant multiplicity spaces at every coordinate $m\geq k$, and
is zero at the remaining coordinates. The face-map identities imply
\begin{equation}\label{eq:direct-shift-identities}
 \begin{split}
 V_{k,i}^\varepsilon(V_{k,l}^\varepsilon)^*
   &=\delta_{il}q_k^\varepsilon,\\
 \sum_{i=1}^{b}(V_{k,i}^\varepsilon)^*V_{k,i}^\varepsilon
   &=q_{k-1}^\varepsilon-b_{k-1}^\varepsilon.
 \end{split}
\end{equation}
These coordinate identities hold in the rectangular matrix spaces
over $A$, since the coordinate representation of $A$ is faithful;
they consequently hold after applying $\rho$.
In particular, the column $V_k^\varepsilon$ is a partial isometry with
final projection $\operatorname{diag}(q_k^\varepsilon,\ldots,q_k^\varepsilon)$.
The products $(V_{k,i}^\varepsilon)^*V_{k,l}^\varepsilon$ are not asserted
to vanish for $i\ne l$.

For the canonical bridge, put
\[
 w_k=q_k^0(1_{H_k}\otimes\rho(s))q_k^1.
\]
Only coordinate $m=k$ contributes, since $s_m$ is supported on
level $m$ and $q_k$ selects level $k$. At this coordinate it sends
the invariant identity vector in copy~1 to that in copy~0. Hence
\begin{equation}\label{eq:direct-bridge}
 w_k^*w_k=b_k^1,\qquad w_kw_k^*=b_k^0.
\end{equation}

\smallskip\noindent\emph{Exact ranks after compression.}
Suppose now that $P$ has rank $N$, commutes with $\rho(e^0)$ and
$\rho(e^1)$, and has commutators of norm at most $\delta$ with
\begin{equation}\label{eq:direct-finite-set}
 u(g)\ (g\in\mathcal F),\quad
 \rho(F^0),\rho(F^1),\quad
 \rho(T_i^\varepsilon)\ (1\leq i\leq b,\ \varepsilon=0,1),\quad
 \rho(s).
\end{equation}
Write $P_k=1_{H_k}\otimes P$. Equation
\eqref{eq:direct-q-commutator} gives
\[
 \|[P_k,q_k^\varepsilon]\|\leq C_q\delta,
 \qquad
 \|[P_k,b_k^\varepsilon]\|,
 \|[P_k,q_k^\varepsilon-b_k^\varepsilon]\|
 \leq(C_q+1)\delta.
\]
Expanding through the three factors defining the rectangular shift
and the bridge gives
\[
 \|P_kV_{k,i}^\varepsilon-V_{k,i}^\varepsilon P_{k-1}\|
 \leq\sqrt D(2C_q+1)\delta,
 \qquad
 \|[P_k,w_k]\|\leq(2C_q+1)\delta.
\]
Set
\[
 C_* =\max\{\sqrt{bD}(2C_q+1),\sqrt2(C_q+1),2C_q+1,1\},
 \qquad 0<\delta<(200C_*)^{-1}.
\]
The column in \eqref{eq:direct-shift-identities}, the two-piece
orthogonal decomposition $q_k^\varepsilon=b_k^\varepsilon+
(q_k^\varepsilon-b_k^\varepsilon)$, and the bridge
\eqref{eq:direct-bridge} all meet the hypothesis of
Lemma~\ref{lem:direct-compression-rank}. Therefore the rounded ranks
\[
 r_k^\varepsilon=\operatorname{rank}
    \chi_{[1/2,\infty)}(P_kq_k^\varepsilon P_k),
 \qquad
 t_k^\varepsilon=\operatorname{rank}
    \chi_{[1/2,\infty)}(P_kb_k^\varepsilon P_k)
\]
satisfy the exact recurrences
\begin{equation}\label{eq:direct-rank-recurrence}
 r_{k-1}^\varepsilon-t_{k-1}^\varepsilon=br_k^\varepsilon\quad(k\geq1),
 \qquad t_k^0=t_k^1\quad(k\geq0).
\end{equation}

Put $P^\varepsilon=P\rho(e^\varepsilon)$. Since $P$ commutes with the
sectors, $r_k^\varepsilon$ is the rounded rank of
$(1\otimes P^\varepsilon)q_k(1\otimes P^\varepsilon)$. The partial trace
formula \eqref{eq:direct-partial-trace} gives
\[
 \operatorname{Tr}((1\otimes P^\varepsilon)q_k(1\otimes P^\varepsilon))
 =d_k^{-1}\operatorname{Tr}(P^\varepsilon z_k)\leq N/d_k.
\]
This is the unnormalized trace of a positive finite-rank operator,
and hence bounds its operator norm. Consequently
\[
 r_k^\varepsilon=0\quad\text{whenever }d_k>2N.
\]
As $d_k\to\infty$, both rank sequences eventually vanish.
Subtracting \eqref{eq:direct-rank-recurrence} and descending from
their common zero tail yields
\begin{equation}\label{eq:direct-equal-ranks}
 r_k^0=r_k^1\quad(k\geq0).
\end{equation}

\smallskip\noindent\emph{Recovering the whole sector dimensions.}
It remains to recover the dimensions of the sectors from these ranks.
Fix $\varepsilon$. Let $X^\varepsilon$ be the orthogonal projection of
$P^\varepsilon$ onto the invariant subspace for conjugation by $u(G)$
on the Hilbert space of Hilbert--Schmidt operators on $\mathcal H$.
The rank of $[P^\varepsilon,u(g)]$ is at most $2N$ and its norm is at
most $\delta$. Equation \eqref{eq:direct-kazhdan} therefore gives
\begin{equation}\label{eq:direct-hs-average}
 \|P^\varepsilon-X^\varepsilon\|_{\mathrm{HS}}
 \leq\kappa^{-1}\max_{g\in\mathcal F}
                       \|[P^\varepsilon,u(g)]\|_{\mathrm{HS}}
 \leq\kappa^{-1}\delta\sqrt{2N}.
\end{equation}
Cesaro means of powers of the symmetric averaging operator
$h^{-1}\sum_{g\in\mathcal F}\operatorname{Ad}u(g)$ converge in
Hilbert--Schmidt norm to this orthogonal projection. Its fixed vectors
are the invariant vectors, because the quadratic form of one minus
the average is the sum of squared displacements, followed by
\eqref{eq:direct-kazhdan}. Each mean applied to $P^\varepsilon$ is a
positive contraction supported under $\rho(e^\varepsilon)$.
Hilbert--Schmidt convergence implies operator-norm convergence, so
$X^\varepsilon$ is a positive compact contraction with this support and
commutes with $u(G)$. Hence
\[
 Q^\varepsilon=\chi_{[1/2,\infty)}(X^\varepsilon)
\]
is a finite-rank projection that commutes with $u(G)$ and lies under
$\rho(e^\varepsilon)$.

We use explicitly the following Hilbert--Schmidt minimization fact.
If $X$ is a positive Hilbert--Schmidt contraction with eigenvalues
$\lambda_j$, then for any finite-rank projection $R$,
\[
 \|X-R\|_{\mathrm{HS}}^2
 =\sum_j\bigl(\lambda_j^2+(1-2\lambda_j)
                            \langle R\xi_j,\xi_j\rangle\bigr).
\]
Each diagonal entry of $R$ belongs to $[0,1]$, so the expression is
minimized term by term by $\chi_{[1/2,\infty)}(X)$. This projection
has finite rank; any zero-eigenspace contributes only nonnegative
diagonal terms. Applying the fact to $X^\varepsilon$ and using the
triangle inequality and \eqref{eq:direct-hs-average} gives
\begin{equation}\label{eq:direct-pq-distance}
 \|P^\varepsilon-Q^\varepsilon\|_{\mathrm{HS}}^2
 \leq4\|P^\varepsilon-X^\varepsilon\|_{\mathrm{HS}}^2
 \leq8\kappa^{-2}\delta^2N.
\end{equation}

On $q_k(H_k\otimes\mathcal H)$ define
\[
 A_k=q_k(1\otimes P^\varepsilon)q_k,
 \qquad B_k=q_k(1\otimes Q^\varepsilon),
 \qquad R_k=\chi_{[1/2,\infty)}(A_k).
\]
Here $B_k$ is a projection, since $Q^\varepsilon$ commutes with $u(G)$
and hence with $q_k$. The nonzero eigenvalues of $Y^*Y$ and $YY^*$,
including their multiplicities, agree for
$Y=q_k(1\otimes P^\varepsilon)$, so $\operatorname{rank}R_k=r_k^\varepsilon$.
The minimization fact, followed by the triangle inequality, yields
\[
 \|R_k-B_k\|_{\mathrm{HS}}
 \leq2\|A_k-B_k\|_{\mathrm{HS}}
 \leq2\|(1\otimes(P^\varepsilon-Q^\varepsilon))q_k\|_{\mathrm{HS}}.
\]
For finite-rank projections $R,B$ one has
\[
 |\operatorname{rank}R-\operatorname{rank}B|
 \leq\operatorname{rank}R+\operatorname{rank}B
                 -2\operatorname{Tr}(RB)=\|R-B\|_{\mathrm{HS}}^2.
\]
Writing $D_\varepsilon=P^\varepsilon-Q^\varepsilon$, we conclude that
\begin{equation}\label{eq:direct-weighted-term}
 |r_k^\varepsilon-\operatorname{rank}B_k|
 \leq4\|(1\otimes D_\varepsilon)q_k\|_{\mathrm{HS}}^2
 =\frac4{d_k}\operatorname{Tr}(D_\varepsilon^2z_k).
\end{equation}
For the last equality, cyclicity against the trace-class operator
$1\otimes D_\varepsilon^2$ gives
\[
 \operatorname{Tr}(q_k(1\otimes D_\varepsilon^2)q_k)
 =\operatorname{Tr}((1\otimes D_\varepsilon^2)q_k)
 =\operatorname{Tr}\bigl(D_\varepsilon^2
       (\operatorname{Tr}_{H_k}\otimes\mathrm{id})(q_k)\bigr),
\]
and then \eqref{eq:direct-partial-trace} applies.
Multiply \eqref{eq:direct-weighted-term} by $d_k$, sum over finite
sets, and take the supremum. The projections $z_k$ are mutually
orthogonal, so every finite sum of them is at most the identity.
All terms are nonnegative, and thus
\begin{equation}\label{eq:direct-weighted-sum}
 \sum_kd_k|r_k^\varepsilon-\operatorname{rank}B_k|
 \leq4\sum_k\operatorname{Tr}(D_\varepsilon^2z_k)
 \leq4\|P^\varepsilon-Q^\varepsilon\|_{\mathrm{HS}}^2.
\end{equation}
No assertion that $\sum_kz_k=1$ is needed.

The finite-dimensional group-invariant range of $Q^\varepsilon$
decomposes into irreducibles in the complete enumeration. The
multiplicity of $\overline\pi_k$ is precisely
$\operatorname{rank}B_k$, so
\[
 \operatorname{rank}Q^\varepsilon
   =\sum_kd_k\operatorname{rank}B_k.
\]
Combining this equality with \eqref{eq:direct-weighted-sum},
$|\operatorname{rank}P^\varepsilon-\operatorname{rank}Q^\varepsilon|
\leq\|P^\varepsilon-Q^\varepsilon\|_{\mathrm{HS}}^2$, and
\eqref{eq:direct-pq-distance} gives
\[
 \left|\operatorname{rank}P^\varepsilon-\sum_kd_kr_k^\varepsilon\right|
 \leq5\|P^\varepsilon-Q^\varepsilon\|_{\mathrm{HS}}^2
 \leq40\kappa^{-2}\delta^2N.
\]
The rank sums are finite because $r_k^\varepsilon=0$ for $d_k>2N$.
Equation \eqref{eq:direct-equal-ranks} now implies
\begin{equation}\label{eq:direct-sector-bound}
 |\operatorname{rank}P^0-\operatorname{rank}P^1|
 \leq80\kappa^{-2}\delta^2N.
\end{equation}

\smallskip\noindent\emph{Passage to quasidiagonal traces.}
Finally, let $\psi_j:A\to M_{N_j}$ be ucp maps witnessing
quasidiagonality of $\mu$: they are asymptotically multiplicative
in norm and $\operatorname{tr}_{N_j}\circ\psi_j\to\mu$ pointwise.
Take Stinespring dilations
$\psi_j(a)=V_j^*\rho_j(a)V_j$ and put $P_j=V_jV_j^*$.
For every fixed contraction $a$,
\[
 \|(1-P_j)\rho_j(a)P_j\|^2
 =\|\psi_j(a^*a)-\psi_j(a)^*\psi_j(a)\|\longrightarrow0.
\]
Applying the same formula to $a^*$ controls the other off-diagonal
corner. Hence the maximum $\delta_j$ of the commutator norms over
the fixed finite set \eqref{eq:direct-finite-set} together with
$e^0,e^1$ tends to zero.

We make the compression commute exactly with the sectors before
applying \eqref{eq:direct-sector-bound}. At one stage, write
$P=P_j$, $e=\rho_j(e^0)$, and $\delta_0=\delta_j$. The positive
finite-rank contraction
\[
 D_P=ePe+(1-e)P(1-e)
\]
satisfies $\|D_P-P\|=\|[P,e]\|\leq\delta_0$ by the two
off-diagonal block formulas. For $\delta_0<1/4$, its spectrum lies
in $[0,\delta_0]\cup[1-\delta_0,1]$. Thus
\[
 \widetilde P=\chi_{[1/2,\infty)}(D_P),
 \qquad\|\widetilde P-P\|\leq2\delta_0<1.
\]
It commutes with both sectors and has rank $N_j$. Indeed, projections
at distance less than one restrict injectively to each other's ranges,
which proves equality of their finite ranks. For every contraction in
the fixed finite set,
\[
 \|[\widetilde P,\rho_j(a)]\|
 \leq\delta_0+2\|\widetilde P-P\|\leq5\delta_0.
\]
For sufficiently large $j$, the required uniform tolerance therefore
holds with $\delta=5\delta_j$. By \eqref{eq:direct-sector-bound},
\[
 N_j^{-1}\left|\operatorname{Tr}
             (\widetilde P_j\rho_j(e^0-e^1))\right|
 \leq2000\kappa^{-2}\delta_j^2.
\]
The operator $\widetilde P_j-P_j$ has rank at most $2N_j$ and norm
at most $2\delta_j$, so its trace pairing with the norm-one operator
$\rho_j(e^0-e^1)$ has absolute value at most $4\delta_jN_j$.
Consequently
\[
 \left|\operatorname{tr}_{N_j}\psi_j(e^0-e^1)\right|
 \leq4\delta_j+2000\kappa^{-2}\delta_j^2\longrightarrow0.
\]
Pointwise trace convergence gives $\mu(e^0)=\mu(e^1)$, and their
sum is $\mu(1)=1$.
\end{proof}

\section{Proof of the main theorem}\label{sec:conclusion}
Fix the data in Theorem~\ref{thm:main} and an integer \(b>R_0\),
and form the algebra in Section~\ref{sec:construction}. It is separable,
unital and RFD. Lemma~\ref{lem:hom-models} supplies its faithful
quasidiagonal trace \(\tau\).
Let \(\mu_1\) be the component supported on copy~1 from
Lemma~\ref{lem:components}, and set
\[
                 \mu_2=\frac13\mu_0+\frac23\gamma.
\]
That lemma gives \(\tau=\frac14\mu_1+\frac34\mu_2\) and
\(\mu_1(e^0)=0,\ \mu_1(e^1)=1\). The trace-balance theorem proved
in the preceding section implies that \(\mu_1\) is not quasidiagonal.
This proves Theorem~\ref{thm:main}, including failure of the face
property.

The same argument shows that \(\mu_0\) is not quasidiagonal.
Neither conclusion is an assertion about a particular unsuccessful
sequence of compressed maps: the trace-balance theorem applies to
every possible sequence of quasidiagonal witnesses.
The direct finite-dimensional models of the faithful trace and the
vanishing normalized size of the top join explain why the convex
combination can behave differently from its components.

\enlargethispage{3\baselineskip}
\section*{Acknowledgment}
OpenAI's ChatGPT assisted with developing and checking the arguments
and with preparing this manuscript.

\bibliographystyle{amsplain}
\bibliography{references}
\end{document}